\documentclass[a4paper, 11pt]{article}

\usepackage{TSeng}
\title{The regular singular inverse problem in differential Galois theory}
\author{Thomas Serafini and Michael Wibmer\thanks{Supported by the Lise Meitner grant M-2582-N32 of the Austrian Science Fund FWF.}}
\date{\today}

\begin{document}

\maketitle

\begin{abstract}
	We show that every linear algebraic group over an algebraically closed field of characteristic zero is the differential Galois group of a regular singular linear differential equation with rational function coefficients. 
	\let\thefootnote\relax
\end{abstract}


\section{Introduction}
Hilbert's 21st problem, also known as the Riemann-Hilbert problem, asks whether there exists a Fuchsian linear differential equation on the Riemann sphere with prescribed singularities and monodromy group. In general, the answer is no, but J. Plemelj showed that the answer is yes if the Fuchsian requirement is slightly relaxed by allowing one singularity to be regular singular (\cite{AnasovBolibruch:TheRiemannHilbertProblem}). In this article, we generalize Plemelj's result from the field of complex numbers to an arbitrary algebraically closed field of characteristic zero. As the monodromy group is only defined over the complex numbers, it is replaced with the differential Galois group in our context. This replacement is justified by Schesinger's density theorem, stating that, for regular singular linear differential equations, the differential Galois group is the Zariski closure of the monodromy group.

\begin{theorem}[Theorem \ref{thm:inv}] \label{theo: main}
	Let $k$ be an algebraically closed field of characteristic zero and let $S=\{p_1,\ldots,p_{d+1}\}$ be a subset of $\P^1(k)$ with $d+1$ elements. Then, for every Zariski closed subgroup $G$ of $\GL_{n,k}$ that can be generated by $d$ elements (as an algebraic group), there exists an ${n\times n}$-matrix $A$ of rational functions over $k$ such that the differential equation $\del(y)=Ay$ has differential Galois group $G$. Furthermore, $A$ can be chosen such that $\del(y)=Ay$ is Fuchsian at $p_1,\ldots,p_d$, regular singular at $p_{d+1}$ and all poles are contained in $S$.	
\end{theorem}
Recall that $\del(y)=Ay$ is \emph{Fuchsian} at $p\in k$ if $A$ has at most a pole of order one at $p$ and that $\del(y)=Ay$ is \emph{regular singular} at $p$ if $\del(y)=Ay$ is Gauge-equivalent to an equation that is Fuchsian at $p$. 
These definitions extend to the point $p=\infty$ via the coordinate transformation $z\mapsto \frac{1}{z}$. The differential equation $\del(y)=Ay$ is \emph{regular singular} if it is regular singular at every point of $\P^1(k)$.

We note that in Theorem \ref{theo: main} for $k=\mathbb{C}$, the topological fundamental group of $\P^1(\C)\smallsetminus S$ is the free group on $d$ generators. Therefore, the monodromy group of $\del(y)=Ay$ can be generated by $d$ elements (as an abstract group) and the differential Galois of $\del(y)=Ay$ can be generated by $d$ elements (as an algebraic group). Furthermore, for $k=\mathbb{C}$, Theorem \ref{theo: main} is well-known. It is essentially the combination of Plemelj's (weak) solution of the Riemann-Hilbert problem with Schlesinger's density theorem. 

As every linear algebraic group over $k$ is the Zariski closure of a finitely generated subgroup and can be embedded into some general linear group, we obtain from Theorem \ref{theo: main}:

\begin{corollary}[Corollary \ref{cor:inv}] \label{cor:intro}
Every linear algebraic group is the differential Galois group of a regular singular differential equation. 
\end{corollary}	
Corollary \ref{cor:intro} generalizes \cite[Theorem 4.3]{Singer:ModuliOfLinearDifferentialEquationsOnTheRiemannSphereWithFixedGaloisGroups} where it was shown, based on a study of moduli of (scalar) linear differential equations, that Corollary \ref{cor:intro}
holds for a large class of linear algebraic groups, including connected and finite linear algebraic groups and groups whose identity component is semisimple or unipotent.

Corollary \ref{cor:intro} can be seen as a sharpening of the solution of the inverse problem in differential Galois theory, which states that every linear algebraic group is the differential Galois group of a linear differential equation with rational function coefficients (\cite{inverse}). Interestingly, this was first proved for $k=\C$ in \cite{TretkoffTretkoff:SolutionOfTheInverseProblem} via Theorem \ref{theo: main}, i.e., by combining Plemelj's (weak) solution of the Riemann-Hilbert problem and Schlesinger's density theorem.

Our strategy to prove Theorem \ref{theo: main} is to deduce it from the case $k=\C$ via a specialization argument, based on a specialization theorem from \cite{fengwib}. This specialization theorem was used in \cite{fengwib} to generalize the solution of the inverse problem in a different direction, by showing that the absolute differential Galois group of the rational function field is a free proalgebraic group.

For $k=\mathbb{C}$ it is an immediate consequence of the Riemann-Hilbert correspondence that the differential Galois group of the family of all regular singular differential equations with singularities contained in $S$, is the free proalgebraic group on $d$ generators (where $|S|=d+1$). It seems natural to expect that this statement is true in full generality, i.e., for any algebraically closed field $k$ of characteristic zero. This question was already raised in \cite{Wibmer:RegSing} and the above theorem can be seen as a small step towards its resolution. Our theorem implies that every linear algebraic group over $k$ that is a quotient of the free proalgebraic group on $d$ generators also is a quotient of the differential Galois group of the family of all regular singular differential equations with singularities in $S$ (Corollary \ref{cor:proalg}).

\medskip

The authors are grateful to Michael Singer and the anonymous referee for helpful comments.

\section{The regular singular inverse problem}

Throughout this article $k$ is an algebraically closed field of characteristic zero and all fields are assumed to have characteristic zero. All rings are assumed to be commutative. For the sake of clarity, we use $t$ for the local variable and $z$ for the global variable. In other words, we write $k((t))$ for the field of formal Laurent series over $k$ and $k(z)$ for the field of rational functions over $k$. For brevity, we use the term ``algebraic group (over $k$)'' for ``affine group scheme of finite type (over $k$)''. In other words, all algebraic groups are assumed to be linear.
A ``proalgebraic group (over $k$)'' is an affine group scheme (over $k$). A ``closed subgroup'' of a proalgebraic group is a closed subgroup scheme.
For a morphism $R\to R'$ of rings and a scheme $X$ over $R$ we denote with $X_{R'}$ the scheme over $R'$ obtained from $X$ by base change via $R\to R'$. For an affine scheme $X$ over a ring $R$, we denote with $R[X]$ the $R$-algebra representing $X$. In particular, for an algebraic group $G$ over $k$, its coordinate ring is denoted with $k[G]$.

\medskip

In this section we prove the results announced in the introduction. After recalling the basics of differential Galois theory and the theory of regular singular differential equations, we study regular singular differential equations under an extension of the constants and under specialization. These preparatory results are then combined to yield the proof of the main theorem.
	
\subsection{Differential Galois theory}
	
In this section we recall a few key definitions and facts from differential Galois theory. For more background see any of the textbooks \cite{SVDP}, \cite{Magdid:Lectures}, \cite{CrespoHajto} or \cite{Saul}.

\medskip

%

Let $(K,\del)$ be a differential field with field of constants $K^\del=\{a\in K|\ \del(a)=0\}$ equal to $k$. A \emph{differential module} $(M,\nabla)$ over $(K,\del)$ is a finite dimensional $K$-vector space together with an additive map $\nabla\colon M\to M$ satisfying $\nabla(am)=\del(a)m+a\nabla(m)$ for $a\in K$ and $m\in M$.	
To a differential equation $\del(y)=Ay$ with $A\in K^{n\times n}$ we can associate the differential module $(K^n,\nabla_A)$, where $\nabla_A=\partial-A\colon K^n\to K^n$ is given by $\nabla_A(\xi)=\del(\xi)-A\xi$.

Conversely, if $(M,\nabla)$ is a differential module and $\underline{e}=(e_1,\ldots,e_n)$ is a $K$-basis of $M$, then
we can write $\nabla(\underline{e})=\underline{e}(-A)$ for some matrix $A\in K^{n\times n}$ so that 
$\nabla(\underline{e}\xi)=\underline{e}\del(\xi)+\nabla(\underline{e})\xi=\underline{e}(\del(\xi)-A\xi)$ for all $\xi\in K^n$. Thus, up to the choice of a basis, $(M,\nabla)$ is of the form $(K^n,\nabla_A)$. The choice of another basis $\underline{e}'$, say $\underline{e}=\underline{e}'F$ with $F\in\GL_n(K)$, leads to another matrix $A'\in K^{n\times n}$ satisfying
\begin{equation}\label{eq:gauge equivalence}
	A'=FAF^{-1}+\del(F)F^{-1} 		
\end{equation}	
Two matrices $A,A'\in K^{n\times n}$ are said to be \emph{Gauge equivalent} (over $K$) if there	exists a matrix $F\in\GL_n(K)$ such that (\ref{eq:gauge equivalence}) is satisfied.

Base change for differential modules works as expected: If $(K',\del')$ is a differential field extension of $(K,\del)$ and $(M,\nabla)$ is a differential module over $(K,\del)$, then $M'=M\otimes_{K}K'$ becomes as differential module over  $(K',\del')$ via $\nabla'(m\otimes a)=\nabla(m)\otimes a+m\otimes\del'(a)$ for $m\in M$ and $a\in K'$. If $(M,\nabla)=(K^n,\nabla_A)$ for some $A\in K^{n\times n}$, then $(M',\nabla')=(K'^n,\nabla_A)$. I.e., in terms of differential equations, this construction simply means that a differential equation $\del(y)=Ay$ over $K$ is considered as a differential equation over $K'$.

\begin{definition} \label{defi:PVring}
For $A \in K^{n \times n}$, a differential $K$-algebra $R$ is a \emph{Picard-Vessiot ring} for $\del(y) = Ay$ if
\begin{enumerate}[(i)]
	\item $R$ is $\del$-simple (i.e., it has no $\del$-stable ideals other than $\{0\}$ and $R$),
	\item there exists a fundamental solution matrix for $\del(y)=Ay$ in $R$, i.e., there exists $Y \in \GL_n(R)$ such that $\del(Y) = AY$,
		\item as a $K$-algebra $R$ is generated by the entries of $Y$ and the inverse of its determinant, i.e., $R = K\big[ Y_{i,j}, \frac{1}{\det(Y)}\big]$.                                                                                                                                         
\end{enumerate}
\end{definition}

Given $A\in K^{n\times n}$, there exists a unique (up to $K$-$\del$-isomorphisms) Picard-Vessiot ring $R/K$ for $\del(y)=Ay$. The \emph{differential Galois group} of $\del(y)=Ay$ or of $R/K$ is then defined as the group of differential automorphisms of $R/K$. More formally, we define $\Gal(R/K)$ as the functor, from the category of $k$-algebras to the category of groups, that associates to a $k$-algebra $T$ the group of differential automorphisms of $R\otimes_k T$ over $K\otimes_k T$, where $T$ is considered as a constant differential ring. The functor $\Gal(R/K)$ is representable by a finitely generated $k$-algebra, i.e., $\Gal(R/K)$ is an algebraic group (over $k$).

Definition \ref{defi:PVring} and the definition of the differential Galois group can be generalized to (potentially infinite) families $(\del(y)=A_iy)_{i\in I}$ of linear differential equations over $K$. For example, for $K=k(z)$ and $S$ a subset of $\P^1(k)$, one can consider the differential Galois group of the family of all regular singular differential equations over $k(z)$ with singularities contained in $S$ (see Section~\ref{subsec: Regular singular differential equations} below for precise definitions). In condition (ii) one requires that for every member of the family there exists a fundamental solution matrix in $R$ and in condition (iii) one allows all entries and the inverse of the determinant of all fundamental solution matrices. 

For families $\Gal(R/K)$ need not be an algebraic group, however, it is still a proalgebraic group (over $k$). Also note that the Picard-Vessiot ring for a family $(\del(y)=A_iy)_{i\in I}$ is generated by the Picard-Vessiot rings of its members $\del(y)=A_iy$.

	\subsection{Regular singular differential equations} \label{subsec: Regular singular differential equations}

	In this section we recall the basic definitions regarding regular singular differential equations and observe that the relevant properties are preserved under an extension of the base field $k$. We also establish a key result (Lemma \ref{lem:basis}) for the proof of our main theorem (Theorem \ref{thm:inv}): We show that if a matrix $A\in k(t)^{n\times n}$ is Gauge equivalent over $k((t))$ to a matrix $A'$ such that $tA'\in k[[t]]^{n\times n}$, then $A$ is already Gauge equivalent over $k(t)$ to such a matrix. More background on regular singular differential equations can be found in \cite[Sections 5 and 6]{SVDP}, \cite{Saul} or \cite[Part~I]{MitschiSauzin}.
	
\medskip	
	
	We first recall the basic definitions regarding formal differential modules, i.e., differential modules over the differential field $(k((t)),\frac{d}{dt})$ of formal Laurent series with coefficients in $k$.

	 A \emph{$k[[t]]$-lattice} $\Lambda$ in a finite dimensional $k((t))$\=/vector space $M$ is a $k[[t]]$-submodule of the form $\Lambda=k[[t]]e_1+\ldots+k[[t]]e_n$, where $e_1,\ldots,e_n$ is a $k((t))$-basis of $M$.  
	A differential module $(M,\nabla)$ over $(k((t)),\frac{d}{dt})$ is 	
	\begin{itemize}
		 \item \emph{regular} if there exists a $k[[t]]$-lattice $\Lambda$ in $M$ such that $\nabla(\Lambda)\subseteq \Lambda$;
		 \item \emph{regular singular} if there exists a $k[[t]]$-lattice $\Lambda$ in $M$ such that $\nabla(\Lambda)\subseteq \frac{1}{t}\Lambda$.
	\end{itemize}
	
For a matrix $A\in k((t))^{n\times n}$, the differential equation $\del(y)=Ay$ is \emph{regular} (or \emph{regular singular}) if the corresponding differential module $(k((t))^n,\nabla_A)$ is regular (or regular singular). Thus $\del(y)=Ay$ is regular if and only if $A$ is Gauge equivalent over $k((t))$ to a matrix $A'\in k[[t]]^{n\times n}$ and $\del(y)=Ay$ is regular singular if and only if $A$ is Gauge equivalent over $k((t))$ to a matrix $A'\in k((t))^{n\times n}$ such that $tA'\in k[[t]]^{n\times n}$.

\medskip

We next discuss our main topic of interest, linear differential equations over $k(z)$, the field of rational functions with coefficients in $k$. Throughout this article we consider $k(z)$ as a differential field with respect to the standard derivation $\frac{d}{dz}$. 

A differential module $(M,\nabla)$ over $k(z)$ defines for every $p\in\P^1(k)=k\cup\{\infty\}$ a formal differential module $(M_p,\nabla_p)$. In detail: For $p\in k$ we set $\widehat{k(z)}^p=k((z-p))$. Then, writing $t=z-p$ so that $\widehat{k(z)}^p=k((t))$, we see that $(k(z),\frac{d}{dz})$ is a differential subfield of $(k((t)),\frac{d}{dt})$. We define the differential module $(M_p,\nabla_p)$ over $\widehat{k(z)}^p=k((t))$ as the base change of $(M,\nabla)$ via the inclusion $k(z)\subseteq \widehat{k(z)}^p$. 

For $p=\infty$, we set $\widehat{k(z)}^p=k((z^{-1}))$. Then, writing $t=z^{-1}$ so that $\widehat{k(z)}^p=k((t))$, we see that $(k(z),\frac{d}{dz})$ is a differential subfield of $(k((t)),-t^2\frac{d}{dt})$. Thus, via base change from  $(k(z),\frac{d}{dz})$ to $(k((t)),-t^2\frac{d}{dt})$ we obtain from $(M, \nabla)$ a differential module $(M_p,\nabla')$ over  $(k((t)),-t^2\frac{d}{dt})$. Set $\nabla_p=-\frac{1}{t^2}\nabla'\colon M_p\to M_p$. Then $(M_p,\nabla_p)$ is a differential module over $(k((t)),\frac{d}{dt})$. 

A point $p\in\P^1(k)$ is \emph{regular} for $(M,\nabla)$ if $(M_p,\nabla_p)$ is regular. Otherwise it is \emph{singular}. A point $p\in\P^1(k)$ is \emph{regular singular} for $(M,\nabla)$ if $(M_p,\nabla_p)$ is regular singular. Finally, $(M,\nabla)$ is \emph{regular singular} if all points $p\in\P^1(k)$ are regular singular for $(M,\nabla)$. As regular points are regular singular, this is equivalent to saying that all singular points are regular singular. 

For $A\in k(z)^{n\times n}$, a point $p\in \P^1(k)$ is \emph{regular} (or \emph{regular singular}) for $\del(y)=Ay$ if $p$ is a regular (or regular singular) point of the corresponding differential module $(k(z)^n,\nabla_A)$. 
So, explicitly, a point $p\in k$ is regular (or regular singular) for $\del(y)=Ay$ if and only if $\del(y)=A(t+p)y$ is regular (or regular singular). On the other hand, $p=\infty$ is a regular (or regular singular) point for $\del(y)=Ay$ if and only if $\del(y)=-\frac{1}{t^2}A(\frac{1}{t})y$ is regular (or regular singular).

The \emph{singularities} or \emph{singular points} of $\del(y)=Ay$ are the points of $\P^1(k)$ that are not regular for $\del(y)=Ay$. A point $p\in k$ is a \emph{pole of $\del(y)=Ay$} if it is a pole of one of the entries of $A$. The point $p=\infty$ is \emph{a pole of $\del(y)=Ay$} if $t=0$ is a pole of $-\frac{1}{t^2}A(\frac{1}{t})$.

 If $p\in\P^1(k)$ is not a pole of $\del(y)=Ay$, then $A\in k[[t]]^{n\times n}$ for $t=z-p$ if $p\in k$ and $-\frac{1}{t^2}A(\frac{1}{t})\in k[[t]]^{n\times n}$ if $p=\infty$. Therefore, $p$ is a regular point for $\del(y)=Ay$. Thus, the singularities of $\del(y)=Ay$ are poles of $\del(y)=Ay$. On the other hand, a pole of $\del(y)=Ay$ need not be a singularity of $\del(y)=Ay$. For example, $p=0$ is a regular point for $\del(y)=\frac{1}{z}y$ (with solution $y(z)=z$). Indeed, for $F=\frac{1}{z}$, we have $F\frac{1}{z}F^{-1}+\del(F)F^{-1}=0$. In the literature, a pole of $\del(y)=Ay$ that is not a singularity of $\del(y)=Ay$ is sometimes referred to as an \emph{apparent singularity}. But note that, according to our terminology, an apparent singularity is not a singularity.

 The differential equation $\del(y)=Ay$ is \emph{Fuchsian} at $p\in\P^1(k)$ if the corresponding matrix has at most a pole of order one, i.e., $tA(t+p)\in k[[t]]^{n\times n}$ for $p\in k$ or $-\frac{1}{t}A(\frac{1}{t})\in k[[t]]^{n\times n}$ for $p=\infty$. If $\del(y)=Ay$ is Fuchsian at $p\in\P^1(k)$, then $\del(y)=Ay$ is regular singular at $p$. Conversely, if $\del(y)=Ay$ is regular singular at $p$, there exists a Gauge transformation over $k((t))$ leading to a matrix with at most a pole of order one.

The differential equation $\del(y)=Ay$ is \emph{regular singular} if all points $p\in \P^1(k)$ are regular singular for $\del(y)=Ay$, i.e., if the differential module $(k(z)^n,\nabla_A)$ is regular singular. 

For $k=\mathbb{C}$, the condition of being regular singular is usually defined via a growth condition on the solutions (see e.g., \cite[Def. 9.11]{Saul} or \cite[Part I, Def. 1.24]{MitschiSauzin}). However, this growth condition is equivalent to the existence of a local meromorphic gauge transformation yielding a transformed matrix that has at most a pole of order one (\cite[Theorem 9.17]{Saul}). The existence of such a local meromorphic gauge transformation is equivalent to the existence of a local formal gauge transformation with the same effect (\cite[Prop. 3.12 and Lemma 3.42]{SVDP}). Thus the above definition of a regular singular differential equation $\del(y)=Ay$ with $A\in k(z)^{n\times n}$ is equivalent to the classical definition for $k=\mathbb{C}$.

The following lemma shows that the conclusion of Theorem \ref{theo: main} remains unfazed under an extension of the base field $k$.

\begin{lemma}\label{lem:basechange}
	Let $k \subseteq k'$ be an inclusion of algebraically closed fields and $S=\{p_1,\ldots,p_{d+1}\}$ a subset of $\P^1(k)$ with $d+1$ elements. If $A\in k(z)^{n\times n}$ is such that $\del(y)=Ay$ has all poles contained in $S$, 	
	is Fuchsian at $p_1,\ldots,p_d$, regular singular at $p_{d+1}$ and has differential Galois group $G$,
	then $\del(y)=Ay$, considered as a differential equation over $k'(z)$, has all poles contained in $S$, 	
	is Fuchsian at $p_1,\ldots,p_d$, regular singular at $p_{d+1}$ and has differential Galois group $G_{k'}$.
\end{lemma}
\begin{proof}
If $p\in\P^1(k')$ is a pole of $\del(y)=Ay$, then $p\in\P^1(k)$. So the poles of $\del(y)=Ay$ do not depend on whether we consider the equation over $k(z)$ or over $k'(z)$. If $\del(y)=Ay$ is Fuchsian (or regular singular) at $p\in \P^1(k)$, when considered as a differential equation over $k(x)$, then $\del(y)=Ay$ is also Fuchsian (or regular singular) at $p$, when considered as a differential equation over $k(x)$.

The statement about the differential Galois group follows from \cite[Lemma 3.2]{fengwib}.
\end{proof}

Our next goal is to show that if a matrix $A\in k(t)^{n\times n}$ is Gauge equivalent over $k((t))$ to a matrix $A'$ such that $tA'\in k[[t]]^{n\times n}$, then $A$ is already Gauge equivalent over $k(t)$ to such a matrix.
For the sake of clarity of the exposition, we separate the proof into a series of lemmas.



\begin{lemma}\label{lem:val}
	If $\Lambda\subseteq k((t))^n$ is a $k[[t]]$-lattice and $E$ is a finite subset of $k((t))^n$, then there exists a non-zero $h\in k((t))$ such that $he\in\Lambda$ for all $e\in E$.
\end{lemma}	
\begin{proof}
	Let $e_1,\ldots,e_n$ be a $k((t))$-basis of $k((t))^n$ such that $\L=k[[t]]e_1+\ldots+k[[t]]e_n$. Fix $e\in E$ and write $e = h_1e_1+...+h_ne_n$ with $h_1,\ldots,h_n\in k((t))$. Let $m_e\in\mathbb{N}$ be such that $t^{m_e}h_i\in k[[t]]$ for $i=1,\ldots,n$.	Then $t^{m_e}e\in\L$. Let $m$ be the maximum of all $m_e$'s. Then $h=t^m$ has the required property.
\end{proof}

\begin{lemma} \label{lemma: scaling lattice}
Let $(M,\nabla)$ be a differential module over $(k((t)),\frac{d}{dt})$ and $\L\subseteq M$ a $k[[t]]$-lattice such that $\nabla(\L)\subseteq \frac{1}{t}\L$. Then $h\L\subseteq M$ is a $k[[t]]$-lattice with $\nabla(h\L)\subseteq \frac{1}{t}h\L$
for any non-zero $h\in k((t))$.
\end{lemma}	
\begin{proof}
 Clearly $h\Lambda$ is a $k[[t]]$-lattice in $M$. Furthermore, $\nabla(h \L) \subseteq \frac{d}{dt}(h)\L + h\nabla(\L) \subseteq \frac{d}{dt}(h)\L + \frac{1}{t}h\L $. So it suffices to show that $\frac{d}{dt}(h)\L\subseteq\frac{1}{t}h\L$. But $\frac{t\frac{d}{dt}(h)}{h}\in k[[t]]$ because the order of $t\frac{d}{dt}(h)$ is the order of $h$ unless $h\in k$. Thus $\frac{d}{dt}(h)\L=\frac{h}{t} \frac{t\frac{d}{dt}(h)}{h} \L \subseteq \frac{h}{t} \L=\frac{1}{t}h\L$.
\end{proof}	

As usual, we denote with $k[t]_{(t)}$ the localization of the ring $k[t]$ at the prime ideal $(t)$, i.e., $k[t]_{(t)}=\left\{\frac{h_1}{h_2}|\ h_1,h_2\in k[t],\ t \text{ does not divide } h_2\right\}$.

\begin{lemma} \label{lem:latt}
Let $e_1,\ldots,e_n$ be a $k(t)$-basis of $k(t)^n$. For $\L=k[[t]]e_1+\ldots+k[[t]]e_n\subseteq k((t))^n$ we have $\L\cap k(t)^n=k[t]_{(t)}e_1+\ldots+k[t]_{(t)}e_n$.
\end{lemma}	
\begin{proof}
	The inclusion $k[t]_{(t)}e_1+\ldots+k[t]_{(t)}e_n\subseteq \L\cap k(t)^n$ follows from $k[t]_{(t)}\subseteq k[[t]]$.
	For the reverse inclusion, consider $m\in\L\cap k(t)^n$. As $e_1,\ldots,e_n$ is a $k((t))$-basis of $k((t))^n$, we have $m=a_1e_1+\ldots+a_ne_n$ for uniquely determined $a_1,\ldots,a_n\in k((t))$. As $m\in \L$, it follows that $a_i\in k[[t]]$ and as $m\in k(t)^n$ is follows that $a_i\in k(t)$. Thus $a_i\in k[[t]]\cap k(t)=k[t]_{(t)}$ for $i=1,\ldots,n$ and so $m\in k[t]_{(t)}e_1+\ldots+k[t]_{(t)}e_n$.
\end{proof}	

%
%
%
%
%

The following lemma is a variant of Nakayama's Lemma. For a proof see e.g., \cite[Chapter~X, Theorem 4.4]{Lang:Algebra}.

\begin{lemma}\label{lem:fingen}
	Let $R$ be a local ring with maximal ideal $\fk m$ and residue field $k=R/\fk m$. If $M$ is a free $R$-module of finite rank, then a tuple of elements from $M$ is an $R$-basis of $M$ if and only if their images are a $k$-basis of $M/\fk m M$.	\qed
\end{lemma}

Recall that the condition of being regular singular is characterized via local formal gauge transformations. The following lemma shows that in fact it is enough to consider rational gauge transformations. We believe this lemma is well-known but we could not find a suitable reference. For statements of a similar spirit in the analytic context see \cite[Lemma 12.1]{Saul} and \cite[Lemma 3.42]{SVDP}.

\begin{lemma}\label{lem:basis}
	Let $A \in k(t)^{n\times n}\subseteq k((t))^{n\times n}$ be such that $\del(y)=Ay$ is regular singular at $0$. Then there
	exists a matrix $F\in \GL_n(k(t))$ such that $A'=FAF^{-1}+\del(F)F^{-1}$ satisfies $tA'\in k[t]_{(t)}^{n\times n}$. In other words, if $A$ is Gauge equivalent over $k((t))$ to a matrix $A'$ satisfying $tA'\in  k[[t]]^{n\times n}$, then $A$ is already Gauge equivalent over $k(t)$ to such a matrix.
\end{lemma}
\begin{proof}
	As $\del(y)=Ay$ is regular singular at $0$, there exists a $k[[t]]$-lattice $\Lambda$ in $k((t))^n$ such that $\nabla_A(\Lambda)\subseteq \frac{1}{t}\Lambda$, where $\nabla_A=\frac{d}{dt}-A$. 
	
%

Let $\underline{e}=(e_1,\ldots,e_n)$ be the standard basis of $k(t)^n$. By Lemma \ref{lem:val} there exists a non-zero $h\in k((t))$ such that $he_i\in \L$ for $i=1,\ldots,n$. Then $\frac{1}{h}\L$ is a $k[[t]]$-lattice in $k((t)))^n$ with $\nabla_A(\frac{1}{h}\L)\subseteq \frac{1}{t}\frac{1}{h}\L$ (Lemma~\ref{lemma: scaling lattice}) and $e_1,\ldots,e_n\in \frac{1}{h}\L$.
Replacing $\L$ with $\frac{1}{h}\L$, we can thus assume that $e_1,\ldots,e_n\in \L$.
	
	Let $\L'$ be the  $k[[t]]$-submodule of $\L$ generated by all $(t\nabla_A)^i(e_j)$, where $i\geq 0$, and $j\in \{1,\ldots,n\}$. Note that, by construction, $t\nabla_A(\L')\subseteq \L'$.

	Recall (e.g. from \cite[Chapter III, Theorem 7.1]{Lang:Algebra}) that a submodule of a free module over a principal ideal domain is free of rank bounded by the rank of the ambient module. As $k[[t]]$ is a principal ideal domain and $\Lambda$ is a free $k[[t]]$-module of rank $n$, it follows that $\L'$ is a free $k[[t]]$-module of rank at most $n$.
	As $\L'$ contains the standard basis $\underline{e}$, we see that in fact $\L'$ is free of rank $n$.
	So $\L'/t\L'$ is a $k$-vector space of dimension $n$.
	
	As the $(t\nabla_A)^i(e_j)$'s generate $\L'$ as a $k[[t]]$-module, Lemma \ref{lem:fingen} shows that we can select $n$ of these generators, say $e_1',\ldots,e_n'$, such that they form a $k[[t]]$-basis of $\L'$. In particular,  $e_1',\ldots,e_n'$ are $k[[t]]$-linearly independent and they lie in $k(t)^n$. Thus $\underline{e}=(e_1',\ldots,e_n')$ is a $k(t)$-basis of $k(t)^n$. 

	Lemma \ref{lem:latt} applied to $e_1',\ldots,e_n'$ and $\L'=k[[t]]e_1'+\ldots+k[[t]]e_n'$ shows that 
	$\L'\cap k(t)^n=k[t]_{(t)}e'_1+\ldots+k[t]_{(t)}e'_n$. As $\L'$ and $k(t)^n$ are stable under $t\nabla_A$, it follows that also $k[t]_{(t)}e'_1+\ldots+k[t]_{(t)}e'_n$ is stable under $t\nabla_A$. That is, if we write $\nabla_A(\underline{e}')=\underline{e'}(-A')$, then $tA'\in k[t]_{(t)}^{n\times n}$. For $F\in\GL_n(k(t))$ with $\underline{e}=\underline{e}'F$, we have $A'=FAF^{-1}+\del(F)F^{-1}$ as desired.
\end{proof}

\subsection{Generating algebraic groups}

In this short section we consider algebraic groups generated by finitely many elements and how this property behaves under base change.

Let $G$ be an algebraic group over $k$ and $g_1,\ldots,g_d\in G(k)$. The intersection of all closed subgroups $H$ of $G$ such that $g_1,\ldots,g_d\in H(k)$, is a closed subgroup of $G$ that also satisfies this property. Thus there exists a smallest closed subgroup of $G$ whose $k$-points contain $g_1,\ldots,g_d$. If this closed subgroup agrees with $G$, we say that $g_1,\ldots,g_d$ \emph{generated $G$ (as an algebraic group)}. 

Let $\Gamma$ be the (abstract) subgroup of $G(k)$ generated by $g_1,\ldots,g_d$. As the Zariski closure of $\Gamma$ is a subgroup of $G(k)$ (\cite[Lemma 2.2.4]{Springer:LinearAlgebraicGroups}), we see that $g_1,\ldots,g_d$ generate $G$ if and only if $\Gamma$ is Zariski dense in $G$. This means that $g_1,\ldots,g_d$ generate $G$ if and only if $f(\gamma)=0$ for all $\gamma\in\Gamma$ implies $f=0$ for $f\in k[G]$.

\begin{lemma}\label{lem:groupgen} Let $k \subseteq k'$ be an inclusion of algebraically closed fields, $G$ an algebraic group over $k$ and $g_1,\ldots,g_d\in G(k)$. Then 
	$g_1,\ldots,g_d$ generate $G$ if and only if $g_1,\ldots,g_d$ generate $G_{k'}$.
\end{lemma}
\begin{proof}
	Let $\Gamma$ be the (abstract) subgroup of $G(k)\subseteq G(k')=G_{k'}(k')$ generated by $g_1,\ldots,g_d$.
	First assume that $g_1,\ldots,g_d$ generate $G$. We have to show that $f'(\gamma)=0$ for all $\gamma\in\Gamma$ implies $f'=0$ for $f'\in k'[G_{k'}]=k[G]\otimes_k k'$. Let $(\lambda_i)_{i\in I}$ be $k$-basis of $k'$. Then $f'$ can be written as $f'=\sum f_i\otimes\lambda_i$ with $f_i\in k[G]$. For $\gamma\in \Gamma$, we have $0=f'(\gamma)=\sum f_i(\gamma)\lambda_i$. As the $\lambda_i$ are $k$-linearly independent and $f_i(\gamma_i)\in k$, it follows that $f_i(\gamma)=0$ for all $\gamma\in\Gamma$ and $i\in I$. Because $g_1,\ldots,g_d$ generate $G$, this implies that $f_i=0$. So $f'=0$ as desired.
	
	Conversely, assume that $g_1,\ldots,g_d$ generate $G_{k'}$. Then an $f'\in k'[G_{k'}]=k[G]\otimes_k k'$ that vanishes on all of $\Gamma$ must be zero. In particular, this holds for $f\in k[G]$ and therefore $g_1,\ldots,g_d$ generate $G$.
	%
	%
	%
\end{proof}
%

\subsection{Specializations}

The proof strategy for our main result is to deduce it from the already known special case $k=\mathbb{C}$ via a specialization argument. In this section we recall the specialization theorem from \cite{fengwib} required for this approach. We also show that the conclusion of Theorem \ref{theo: main} is preserved under specialization in an appropriate sense.


We begin by recalling some terminology from \cite{fengwib}. Let $\l B$ be a $k$-algebra (considered as a constant differential ring). The specializations we will be considering are morphisms $c\in \Hom_k(\l B,k)$, i.e., $c\colon \l B\to k$ is a morphism of $k$-algebras. We consider $\l B[z]$ as a differential ring with respect to the derivation $\frac{d}{dz}$. For a monic polynomial $f\in \l B[z]$ a specialization $c\colon \l B\to k$ extends to a morphism $c\colon \l B[z]_f\to k(z)$ of differential $k$-algebras via $c(z)=z$.
For a matrix $\l A\in \l B[z]_f^{n\times n}$ we denote with $A^c\in k(z)^{n\times n}$ the matrix obtained from $\l A$ by applying $c\colon \l B[z]_f\to k(z)$ to the entries of $\l A$. Similarly, for $a\in \l B$ we will write $a^c$ rather than $c(a)$.

 Assuming that $\l B$ is an integral domain, so that we can consider the algebraic closure $K$ of the field of fractions of $\l B$, the specialization problem is about comparing the differential Galois group of the generic equation $\del(y)=\l A y$ (over $K(z)$) with the differential Galois group of the specialized equations $\del(y)=A^cy$ (over $k(z)$). To capture the idea of specializing differential Galois groups, we need to consider group schemes over $\l B$ and we also need some more terminology.

 Let $\l B$ be a $k$-algebra and let $\l Q$ be a differential $\l B$-algebra. (The most relevant case for us is $\l Q=\l B[z]_f$.) For a differential $\l Q$-algebra $\l R$, we define $\underline{\Aut}(\l R/\l Q)$ as the functor form the category of $\l B$-algebras to the category of groups given by $\underline{\Aut}(\l R/\l Q)(\l T) = \Aut^\del(\l R \otimes_{\l B} \l T / \l Q \otimes_{\l B} \l T)$, where $\l T$ is considered as a constant differential ring. On morphisms, $\underline{\Aut}(\l R / \l Q)$ is defined by base change.
An \emph{action} of an affine group scheme $\l G$ over $\l B$ on $\l R/\l Q$ is a morphism of group functors $\l G \to \underline{\Aut}(\l R/\l Q)$.

%
%
%
%

Let $\l B[\l G]$ denote the coordinate ring of $\l G$, i.e., the $\l B$-algebra representing $\l G$. By \cite[Lemma 3.3]{fengwib}, to specify an action of $\l G$ on $\l R / \l Q$ is equivalent to specifying a morphism (the coaction) $\rho\colon \l R \to \l R \otimes_{\l B} \l B[\l G]$ of $\l Q$-$\del$-algebras such that 
\[\begin{tikzcd} \l R \otimes_{\l B} \l B[\l G] \otimes_{\l B}\l B[\l G] & \l R \otimes_{\l B} \l B[\l G] \arrow[swap]{l}{\rho \otimes \text{id}}\\ \l R \otimes_{\l B} \l B[\l G] \arrow{u}{\text{id}\otimes \Delta} & \l R \arrow[swap]{u}{\rho} \arrow{l}{\rho} \end{tikzcd}\qquad \text{ and }\qquad \begin{tikzcd}  \l R \otimes_{\l B} \l B[\l G] \arrow[swap]{d}{\text{id} \otimes \varepsilon} & \l R \arrow[swap]{l}{\rho} \\ \l R \otimes_{\l B} \l B \arrow[ur, equal]\end{tikzcd}\]
 commute, where $\Delta$ and $\varepsilon$ denote the comultiplicaiton and counit of the Hopf-algebra $\l B[\l G]$.

\begin{definition} \label{def: differential torsor}
Assume, as above, that $\l G$ acts on $\l R / \l Q$. Then $\l R / \l Q$ is a \emph{differential $\l G$-torsor} if the morphism $\l R \otimes_{\l Q} \l R \to \l R \otimes_{\l B} \l B[\l G]$, $\ a \otimes b \mapsto (a \otimes 1) \cdot \rho(b)$ is an isomorphism. Moreover, for $\l A \in \l Q^{n \times n}$, if there exists a matrix $\l Y \in \GL_n(\l R)$ such that
	\begin{itemize}
		\item $\del(\l Y) = \l A \l Y$,
		\item $\l R = \l Q\lb \l Y_{i,j}, \frac{1}{\det(\l Y)}\rb$, and
		\item for every $\l B$-algebra $\l T$ and every $g \in \l G(\l T)$, there exists a matrix $M_g \in \GL_n(\l T)$ such that $g(\l Y \otimes 1) = (\l Y \otimes 1)(1 \otimes M_g)$,
	\end{itemize}
	 then $\l R / \l Q$ is a \emph{differential $\l G$-torsor for $\del(y) = \l Ay$}.
\end{definition}

The definition of a differential torsor can be interpreted geometrically: With $\l Z = \Spec(\l R)$, the coaction $\rho\colon \l R \to \l R \otimes_{\l B} \l B[\l G]=\l R\otimes_{\l Q}\l Q\otimes_{\l B }\l B[\l G]$ induces a group action $\l Z \times_{\l Q} \l G_{\l Q} \to \l Z$ (in the category of $\l Q$-schemes) and $\l R/\l Q$ is a differential $\l G$-torsor if and only if $\l Z \times_{\l Q} \l G_{\l Q} \to \l Z \times_{\l Q} \l Z$ is an isomorphism.

Definition \ref{def: differential torsor} mimics and generalizes the definition of a Picard-Vessiot ring.
In particular, if $R/K$ is a Picard-Vessiot ring for $\del(y)=Ay$ (Definition \ref{defi:PVring}), then $R/K$ is a 
differential $\Gal(R/K)$-torsor for $\del(y) = A y$.

The following specialization theorem is the key to our specialization argument.

\begin{theorem}[\cite{fengwib}, Theorem 4.26] \label{thm:spe}
	Assume that
	\begin{itemize}
		\ibul $\l B$ is a finitely generated $k$-algebra and an integral domain,
		\ibul $K$ is the algebraic closure of the field of fractions of $\l B$,
		\ibul $\l G$ is an affine group scheme of finite type over $\l B$,
		\ibul $f \in \l B[z]$ is a monic polynomial,
		\ibul $\l A \in \l B[z]^{n \times n}_f$,
		\ibul $\l R/\l B[z]_f$ is a differential $\l G$-torsor for $\del(y) = \l A y$ such that $\l R$ is flat over $\l B[z]_f$,
		\ibul $\l R \otimes_{\l B[z]_f} K(z)$ is $\del$-simple.
	\end{itemize}
	Then there exists a $c \in \Hom_{k}(\l B, k)$ such that $R^c=\l R \otimes_{\l B[z]_f} k(z)$ is a Picard-Vessiot ring for $\del(y)=A^cy$ with differential Galois group $\l G_k$, where the involved base change is $c\colon \l B \to k$.
	
\end{theorem}

Note that the assumptions in Theorem \ref{thm:spe} imply that $R^{\textup{gen}}=\l R \otimes_{\l B[z]_f} K(z)$ is a Picard-Vessiot ring for $\del(y)=\l A y$ (over $K(z)$) with differential Galois group $\l G_K$. The following lemma shows that, starting from an inclusion $k\subseteq k'$ of algebraically closed fields and a differential equation $\del(y)=Ay$ over $k'(z)$, the technical conditions of Theorem \ref{thm:spe} can always be achieved.

%

\begin{lemma}[\cite{fengwib}, Lemma 3.11] \label{lem:spread}
	Let $k \subseteq k'$ be an inclusion of algebraically closed fields and let $A\in k'(z)^{n\times n}$. 
	 Then there exist
	
	\begin{itemize}
		\item [$\bullet$] a finitely generated $k$-subalgebra $\l B$ of $k'$,
		\item[$\bullet$] a monic polynomial $f \in \l B[z]$,
		\item[$\bullet$] an affine group scheme $\l G$ of finite type over $\l B$, and
		\item[$\bullet$] a differential $\l G$-torsor $\l R/\l B[z]_f$,
	\end{itemize}
	such that
	\begin{itemize}
		\ibul $A  \in \l B[z]_f^{n\times n}$,
		\ibul $\l R$ is flat over $\l B[z]_f$ and $\l R / \l B[z]_f$ is a differential $\l G$-torsor for $\del(y) = Ay$
		\ibul $\l R \otimes_{\l B[z]_f} K(z)$ is $\del$-simple, where $K$ the algebraic closure of the field of fractions of $\l B$.
	\end{itemize}
	Moreover, for a fixed finitely generated $k$-algebra $\l B_0$ and monic $f_0 \in \l B_0[z]$, $\l B$ and $f$ can be chosen such that $\l B_0 \subseteq \l B$ and $f_0 $ divides $f$.
	Furthermore, if there exists an algebraic group $G$ over $k$ such that the differential Galois group of $\del(y)=Ay$ (over $k'(z)$) is of the form $G_{k'}$, then $\l G$ can be chosen to be $G_{\l B}$.
\end{lemma}

%
%
%
%
%


We next study regular singular differential equations under specialization, first locally, then globally.

\begin{lemma}\label{lem: specialize regsing at 0}
	Let $k\subseteq k'$ be an inclusion of algebraically closed fields and let $A \in k'(t)^{n \times n}\subseteq k'((t))^{n\times n}$ be such that $\del(y)=Ay$ is regular singular at $0$. Then there exists a finitely generated $k$\=/subalgebra $\l B$ of $k'$ and a monic polynomial $f \in \l B[t]$ such that $A \in \l B[t]_f^{n \times n}$ and $\del(y)=A^c y$ is regular singular at $0$ for every $c \in \Hom_k(\l B, k)$.	
\end{lemma}
\begin{proof}
By Lemma \ref{lem:basis}, there exists a matrix $F\in\GL_n(k'(t))$ such that the matrix
		$t(FAF^{-1}+\del(F)F^{-1})$ has entries in $k'[t]_{(t)}$. We may choose a monic polynomial $f\in k'[t]$ such that $fA$ and $fF$ have entries in $k'[t]$. Let $\l B$ be the $k$\=/subalgebra of $k'$ generated by the coefficients of $f$ and all coefficients of all entries of $fA$ and $fF$. Then $A,F\in \l B[t]_f^{n\times n}$.
		As $\det(F)\in \l B[t]_f$ is non-zero, we may write $\det(f)=\frac{h}{f^m}$ for some non-zero $h\in \l B[t]$. Adding the inverse of the leading coefficient of $h$ to $\l B$ and replacing $f$ with the product of $f$ with $h$ and the inverse of the leading coefficient of $h$, we may assume that $A, F, F^{-1}\in \l B[t]_f^{n\times n}$. 
		
		As $t(FAF^{-1}+\del(F)F^{-1})$ has entries in $k'[t]_{(t)}$, there exist non-zero elements $a_1,
		\ldots,a_m\in k'$ and a matrix $B\in k'[t]$ such that 
		\begin{equation} \label{eq: for regular}
					(t-a_1)\ldots(t-a_m)t(FAF^{-1}+\del(F)F^{-1})=B.
		\end{equation}		
		Replacing $\l B$ with the $k$-subalgebra of $k'$ generated by $\l B$, the coefficients of the entries of $B$ and $a_1,\ldots,a_m,a_1^{-1},\ldots,a_m^{-1}$, identity (\ref{eq: for regular}) becomes an identity in $\l B[t]_f$. Thus, for any $c \in \Hom_k(\l B, k)$, we can apply the morphism $c\colon \l B[t]_f\to k(t)$ of differential $k$-algebras to (\ref{eq: for regular}) to obtain
		$$
		(t-a_1^c)\ldots(t-a_m^c)t(F^cA^c(F^c)^{-1}+\del(F^c)(F^c)^{-1})=B^c,
		$$
		an identity in $k(t)$. As $B^c\in k[t]^{n\times n}$ and $a_1^c,\ldots,a_m^c$ are non-zero, this shows that $\del(y)=A^cy$ is regular singular at $0$.	
\end{proof}


%
%

The following lemma shows that the conclusion of Theorem \ref{theo: main} is preserved under specialization.

\begin{lemma}\label{lem:regsing}	
	Let $k\subseteq k'$ be an inclusion of algebraically closed fields and let $S=\{p_1,\ldots,p_{d+1}\}$ be a subset of $\P^1(k)\subseteq\P^1(k')$ with $d+1$ elements. Furthermore, let $A \in k'(z)^{n \times n}$ be such that $\del(y)=Ay$ has all poles contained in $S$, is Fuchsian at $p_1,\ldots,p_d$ and regular singular at $p_{d+1}$. Then there exists a finitely generated $k$-subalgebra $\l B \subseteq k'$ and a monic polynomial $f \in \l B[z]$ such that $A \in \l B[z]_f^{n \times n}$ and for every $c \in \Hom_k(\l B, k)$ the differential equation $\del (y) =A^c y$ has all poles contained in $S$, is Fuchsian at $p_1,\ldots,p_d$ and regular singular at $p_{d+1}$.	
\end{lemma}

\begin{proof}
	We first assume that $\infty\notin S$. By assumption,
	$B=(z-p_1)\ldots(z_1-p_d)(z-p_{d+1})^eA\in k'[z]^{n\times n}$ for some integer $e\geq 1$.	
		 Let $\l B_0$ be the $k$\=/subalgebra of $k'$ generated by the coefficients of all entries of $B$ and set $f=(z-p_1)\ldots(z_1-p_d)(z-p_{d+1})^e\in k[z]\subseteq\l B_0[z]$.
	Then $A\in \l B_0[z]_{f}^{n\times n}$ and for any $c \in \Hom_k(\l B_0, k)$ the poles in $k$ of the differential equation $\del (y) =A^c y$ are contained in $S$.
	
	For a non-zero rational function $h\in \l B_0[z]_{f}\subseteq k'(z)$, when written as $h=\frac{h_1}{f^m}$, with $h_1\in\l B_0[z]$, the rational function $-\frac{1}{t^2}h(\frac{1}{t})\in k'(t)$ has a pole at $t=0$ if and only if $\deg(h_1)-m(d+e)+2>0$. Thus, if $-\frac{1}{t^2}h(\frac{1}{t})$ does not have a pole at $t=0$, for any specialization $c\in \Hom_k(\l B_0, k)$, the rational function $-\frac{1}{t^2}h^c(\frac{1}{t})$ does not have a pole at $t=0$. So, as $p=\infty$ is not a pole of $\del(y)=Ay$, 
	$p=\infty$ is also not a pole of $\del(y)=A^cy$ for any $c\in\Hom_k(\l B_0,k)$. In summary, for any $c\in\Hom_k(\l B_0,k)$, the poles of $\del(y)=A^cy$ are contained in $S$.
	
%
%
%
%
	 The differential equation $\del(y)=A(t+p_{d+1})y$ is regular singular. Thus, by Lemma \ref{lem: specialize regsing at 0}, there exists a finitely generated $k$-subalgebra $\l B$ of $k'$ and a monic polynomial $f_1\in \l B[t]$ such that $A(t+p_{d+1})\in \l B[t]_{f_1}^{n\times n}$ and $\del(y)=A(t+p_{d+1})^cy$ is regular singular for all $c\in\Hom_k(\l B,k)$. Enlarging $\l B$ if necessary, we may assume that $\l B_0\subseteq \l B$. Furthermore, replacing $f_1\in \l B[t]$ by $f_1(t)f(t+p_{d+1})\in \l B[t]$, we may assume that $f(t+p_{d+1})$ divides $f_1$ in $\l B[t]$.
	 We claim that $\l B$ and $f\in\l B[z]$ have the required properties. 
	 
	 Let $c\in\Hom_k(\l B,k)$. Note that $\l B_0[t+p_{d+1}]_{f(t+p_{d+1})}\subseteq \l B[t]_{f_1}$ and $h(t+p_{d+1})^c=h^c(t+p_{d+1})$ for $h\in \l B_0[z]_{f}$. So  $A(t+p_{d+1})^c=A^c(t+p_{d+1})$.
	  As $\del(y)=A(t+p_{d+1})^cy$ is regular singular, this shows that $\del (y)=A^cy$ is a regular singular at $p_{d+1}$. As explained in the second paragraph of this proof, the poles of $\del (y)=A^cy$ are contained in $S$.
	 Finally, from $(z-p_1)\ldots(z_1-p_d)(z-p_{d+1})^eA\in \l B[z]$, we obtain $(z-p_1)\ldots(z_1-p_d)(z-p_{d+1})^eA^c\in k[z]$, showing that $\del (y)=A^cy$ is Fuchsian at $p_1,\ldots,p_{d}$.
	 
	 In case $\infty\in S$, a slight modification of the above proof yields the result. In detail, if $\infty$ is among the Fuchsian points, say $p_1=\infty$, then, with $f=(z-p_2)\ldots(z-p_d)(z-p_{d+1})^e\in k[z]$, we can proceed as above, noting that $-\frac{1}{t}A(\frac{1}{t})\in k'[[t]]^{n\times n}$ implies $-\frac{1}{t}A^c(\frac{1}{t})\in k[[t]]^{n\times n}$ for any $c\in\Hom_k(\l B_0,k)$ (cf. the second paragraph of this proof).
	 
	 If $\infty$ is only regular singular, i.e., $p_{d+1}=\infty$, then, with $f=(z-p_1)\ldots(z-p_d)\in k[z]$ we can proceed similarly to the above, but using that, as $\del(y)=-\frac{1}{t^2}A(\frac{1}{t})y$ is regular singular, there exists, by Lemma~\ref{lem: specialize regsing at 0}, a finitely generated $k$-subalgebra $\l B$ of $k'$ and a monic polynomial $f_1\in\l B[t]$ such that $-\frac{1}{t^2}A(\frac{1}{t})\in \l B[t]_{f_1}^{n\times n}$ and $\del(y)=(-\frac{1}{t^2}A(\frac{1}{t}))^cy$ is regular singular for every $c\in\Hom_k(\l B,k)$. Enlarging $\l B$ if necessary, we can assume that $\l B_0\subseteq \l B$.  Writing $f(\frac{1}{t})=\frac{cf_2}{t^m}$ with $c\in k$ non-zero and $f_2\in k[t]$ monic and replacing $f_1$ by $f_1tf_2$, we can assume that $t$ and $f_2$ divide $f_1$ in $\l B[t]$. As $\l B_0[\frac{1}{t}]\subseteq \l B_0[t]_t\subseteq \l B[t]_{f_1}$ and $\frac{1}{f(\frac{1}{t})}=\frac{c^{-1}t^m}{f_2}\in \l B[t]_{f_1}$, we see that $\l B_0[\frac{1}{t}]_{f(\frac{1}{t})}\subseteq \l B[t]_{f_1}$. Furthermore, for $h\in\l B_0[z]_f$, we have $-\frac{1}{t^2}h(\frac{1}{t})\in\l B[t]_{f_1}$ and $\left(-\frac{1}{t^2}h(\frac{1}{t})\right)^c=-\frac{1}{t^2}h^c(t)$ for any $c\in\Hom_k(\l B,k)$. Thus $\left(-\frac{1}{t^2}A(\frac{1}{t})\right)^c=-\frac{1}{t^2}A(\frac{1}{t})^c$ and  $\del(y)=-\frac{1}{t^2}A(\frac{1}{t})^cy$ is regular singular, i.e., $\del(y)=A^cy$ is regular singular at $\infty$ for every $c\in\Hom_k(\l B,k)$. It follows that $\l B$ and $f$ have the required properties.
\end{proof}

\subsection{Proof of the main result}

We are now prepared to prove our main theorem.

\begin{theorem}\label{thm:inv}
	Let $G$ be a closed subgroup of $\GL_{n,k}$ that can be generated by $d$ elements and let $S=\{p_1,\ldots,p_{d+1}\}$ be a subset of $\P^1(k)$ with $d+1$ elements. Then there exists a matrix $A\in k(z)^{n\times n}$ such that the differential equation $\del(y)=Ay$ has all poles contained in $S$, is Fuchsian at $p_1,\ldots,p_d$, regular singular at $p_{d+1}$ and has differential Galois group $G$.
\end{theorem}

\begin{proof}
	 We roughly follow the short solution to the inverse problem from \cite[Section 5.2]{fengwib}. The algebraic group $G$ (together with the closed embedding into $\GL_{n,k}$) descends to a finitely generated field $k_0$, i.e., there exist a subfield $k_0$ of $k$, finitely generated over $\mathbb{Q}$, and a closed subgroup $G_0$ of $\GL_{n,k_0}$ such that $(G_0)_k=G$ (as closed subgroups of $\GL_{n,k}$).
	 
	 Let $g_1,\ldots,g_d\in G(k)$ generate $G$. Enlarging $k_0$ if necessary, we can assume that $g_1,\ldots,g_d\in G_0(k_0)\subseteq G_0(k)=G(k)$ and $S\subseteq \P^1(k_0)$.

	 
	 Let $k_1\subseteq k$ be the algebraic closure of $k_0$ and set $G_1=(G_0)_{k_1}$. Then $G_1$ is a closed subgroup of $\GL_{n,k_1}$. 
	 As $k_1$ is countable, there exists an embedding $k_1\hookrightarrow \C$ (that we fix).
	 
	 Because $g_1,\ldots,g_d\in G_0(k_0)\subseteq G_0(k_1)=G_1(k_1)$ generate $G=(G_1)_k$, it follows from Lemma~\ref{lem:groupgen} that $g_1,\ldots,g_d$ generate $G_1$. The same lemma, but applied in the reverse direction, shows that $(G_1)_{\C}$ is generated by $g_1,\ldots,g_d$. In particular, $(G_1)_{\C}$ is a closed subgroup of $\GL_{n,\C}$ that can be generated by $d$ elements.
	 
	 As noted in the introduction, Theorem \ref{thm:inv} holds for $k=\C$ as a consequence of the (weak) solution of the Riemann-Hilbert problem (see \cite[Theorem 3.2.1]{AnasovBolibruch:TheRiemannHilbertProblem}, \cite[Theorem 12.4]{Saul} or \cite[Part I, Theorem 4.49]{MitschiSauzin}) and Schlesinger's density theorem (see e.g., \cite[Part I, Theorem 2.28]{MitschiSauzin}).
	 Thus, there exists a matrix $A\in \mathbb{C}(z)^{n\times n}$ such that $\del(y)=Ay$ has poles contained in $S$, is Fuchsian at $p_1,\ldots,p_d$, regular singular at $p_{d+1}$ and has differential Galois group $(G_1)_{\mathbb{C}}$. 
	 
%
%
%

	Applying Lemma \ref{lem:regsing} to the inclusion $k_1\subseteq \C$ and the differential equation $\del(y)=Ay$ yields a finitely generated $k_1$-subalgebra $\l B_0$ of $\C$ and a monic $f_0\in \l B_0[z]$ with $A\in \l B_0[z]_{f_0}^{n\times n}$ such that $\del(y)=A^cy$ has all poles contained in $S$, is Fuchsian at $p_1,\ldots,p_d$ and regular singular at $p_{d+1}$ for all $c\in\Hom_{k_1}(\l B_0,k_1)$. By Lemma~\ref{lem:spread}, there exists a finitely generated $k_1$-subalgebra $\l B$ of $\mathbb{C}$ containing $\l B_0$, a monic polynomial $f \in \l B[z]$ such that $f_0$ divides $f$, and a differential $(G_1)_{\l B}$\=/torsor $\l R/\l B[z]_f$ for $\del(y)=Ay$ such that $\l R$ is flat over $\l B[z]_f$ and $\l R\otimes_{\l B[z]_f}K(z)$ is $\del$-simple, where $K$ is the algebraic closure of the field of fractions of $\l B$.
	By Theorem \ref{thm:spe}, there exists a $c \in \Hom_{k_1}(\l B, k_1)$ such that $R^c = \l R \otimes_{\l B[z]_f} k_1(z)$ is a Picard-Vessiot ring for $\del(y)=A^cy$ (over $k_1(z)$) with differential Galois group $((G_1)_{\l B})_{k_1}=G_1$.
	By construction of $\l B_0$ and $f_0$, and using that $A\in \l B_0[z]_{f_0}^{n\times n}\subseteq \l B[z]_f^{n\times n}$, we see that $\del(y)=A^cy$ has all poles contained in $S$, is Fuchsian at $p_1,\ldots,p_d$ and regular singular at $p_{d+1}$.
	
	Finally, by Lemma \ref{lem:basechange}, the equation $\del(y)=A^cy$, considered as a differential equation over $k(z)$, has the required properties, i.e., $\del(y)=A^cy$ (over $k(z)$) has all poles contained in $S$, is Fuchsian at $p_1,\ldots,p_d$, regular singular at $p_{d+1}$ and has differential Galois group $(G_1)_k=G$.	
%
\end{proof}



	\begin{corollary}\label{cor:inv}
		Any algebraic group over $k$ is the differential Galois group of a regular singular differential equation $\del (y) = A y$ over $k(z)$.
	\end{corollary}
	\begin{proof}
	As any (affine) algebraic group is isomorphic to a closed subgroup of some general linear group (\cite[Cor. 4.10]{Milne:AlgebraicGroups}) and every algebraic group (in characteristic zero) is finitely generated (\cite{SVDP}, Lemma 5.13), this follows from Theorem \ref{thm:inv}.
	\end{proof}

\subsection{An interpretation in terms of proalgebraic groups}

In this final section we offer an interpretation of our main result (Theorem \ref{thm:inv}) in terms of proalgebraic groups.

Recall that the \emph{free proalgebraic group} $\Gamma_d$ on $d$ generators (over $k$) is characterized by the following universal property of the map $\iota\colon X\to \Gamma_d(k)$ from a $d$-element set $X$ into the $k$-points of $\Gamma_d$: For every proalgebraic group $G$ and every map $\varphi\colon X\to G(k)$ there exists a unique morphism $\phi\colon \Gamma_d\to G$ of proalgebraic groups such that
$$
\begin{tikzcd}  X \arrow[d, "\varphi"'] \arrow{r}{\iota} & \Gamma_d(k) \arrow[ld, dashrightarrow, "\phi"]  \\  G(k)  \end{tikzcd}
$$
commutes. The proalgebraic group $\Gamma_d$ can be constructed as the fundamental group of the neutral tannakian category of all (finite dimensional, $k$-linear) representations of the (abstract) free group $F_d$ on $d$ generators. In other words, $\Gamma_d$ is the proalgebraic completion (or proalgebraic hull) of $F_d$. (See \cite{proalg} for more background on free proalgebraic groups.)

Let $S$ be a subset of $\P^1(k)$ with $d+1$ elements and let $\Gamma_S$ be the differential Galois group of the family of all regular singular differential equations over $k(z)$ with singularities contained in $S$.
For $k=\C$, it is an immediate corollary of the Riemann-Hilbert correspondence (\cite[Theorem~6.15]{SVDP}) that $\Gamma_S$ is isomorphic to $\Gamma_d$. It therefore seems natural to expect that $\Gamma_S$ is isomorphic to $\Gamma_d$ in general. This question was already raised in \cite[Section~4.1]{Wibmer:RegSing}. While we are far from being able to show that indeed $\Gamma_S$ is isomorphic to $\Gamma_d$, Theorem \ref{thm:inv} can be seen as a small step into the right direction. It implies that every algebraic quotient of $\Gamma_d$ is also a quotient of $\Gamma_S$.


\begin{corollary}\label{cor:proalg}
	Let $G$ be an algebraic group over $k$. If $G$ is a quotient of $\Gamma_d$, then $G$ is also a quotient of $\Gamma_S$.
\end{corollary}
\begin{proof}
If $G$ is a quotient of $\Gamma_d$, then $G$ can be generated by $d$ elements (\cite[Lemma 2.16]{proalg}). After embedding $G$ into some $\GL_{n,k}$, we can apply Theorem \ref{thm:inv} to find a regular singular differential equation $\del(y)=Ay$ with differential Galois group $G$. The Picard-Vessiot ring $R$ of $\del(y)=Ay$ is canonically embedded in the Picard-Vessiot ring $R_S$ of the family of all regular singular differential equations with singularities in $S$. By the second fundamental theorem of differential Galois theory (see e.g., \cite[Theroem 2.11]{AmanoMasuokaTakeuchi}), the differential Galois group of $R$ is a quotient of the differential Galois group of $R_S$, i.e., $G$ is a quotient of $\Gamma_S$ as desired.
\end{proof}

%
%

\printbibliography

Thomas Serafini, Sorbonne Université, IMJ-PRG, 4 place Jussieu, 
75005 Paris, France, \texttt{tserafini@dma.ens.fr} 

\medskip

Michael Wibmer, School of Mathematics, University of Leeds, LS2 9JT, Leeds, United Kingdom, \texttt{m.wibmer@leeds.ac.uk}

\end{document}